\pdfoutput=1
\documentclass[11pt]{article}

\usepackage[a4paper,margin=1.1in]{geometry}
\usepackage[T1]{fontenc}
\usepackage[utf8]{inputenc}
\usepackage{lmodern}
\usepackage{microtype}
\usepackage[authoryear,longnamesfirst]{natbib}
\usepackage{amsmath,amssymb,amsthm,mathtools}
\usepackage{url}
\usepackage[colorlinks=true,linkcolor=blue,citecolor=blue,urlcolor=blue]{hyperref}
\hypersetup{
  pdftitle={Counterexamples to a multivariable matrix Young conjecture},
  pdfauthor={Zhekai Pang}
}

\theoremstyle{plain}
\newtheorem{conjecture}{Conjecture}
\newtheorem{theorem}{Theorem}
\newtheorem{lemma}{Lemma}

\theoremstyle{remark}
\newtheorem{remark}{Remark}

\title{Counterexamples to a multivariable matrix Young conjecture}
\author{Zhekai Pang\thanks{Universitat Pompeu Fabra, Barcelona, Spain. Email: \texttt{zhekai.pang@upf.edu}.}}
\date{}

\begin{document}

\maketitle

\begin{abstract}
We disprove Conjecture~5.1 of Lin concerning a multivariable extension of Ando's matrix Young inequality for positive semidefinite matrices. For every integer $m\geq4$, the conjectured eigenvalue inequality fails at the largest eigenvalue for $2\times2$ real rank-one orthogonal projections. We also give a $3\times3$ real positive semidefinite counterexample for $m=3$, where the failure occurs at the second eigenvalue.
\end{abstract}



\section{Introduction}\label{sec:intro}

Young's inequality is a fundamental scalar estimate that has prompted a substantial literature on matrix, norm, and trace analogues. Ando established a two-factor version in the Loewner order \citep{Ando1995}. Manjegani studied H\"older and Young inequalities for traces of operators \citep{Manjegani2007}, Bhatia and Kittaneh revisited a related matrix arithmetic--geometric mean problem \citep{BhatiaKittaneh2008}, Ko\v{s}em developed further matrix versions \citep{Kosem2009}, and Drury settled the singular-value question posed by Bhatia and Kittaneh \citep{Drury2012}.

Let $M_n$ denote the set of $n$ by $n$ complex matrices. For $X\in M_n$, write
\[
 |X|=(X^*X)^{1/2}.
\]
For Hermitian matrices $X$ and $Y$, the notation $X\preceq Y$ means that $Y-X$ is positive semidefinite. If $A,B\in M_n$ and $p,q>1$ satisfy $1/p+1/q=1$, Ando's matrix Young inequality asserts that there is a unitary matrix $U$ such that
\begin{equation}\label{eq:ando}
 U|AB^*|U^*\preceq \frac1p|A|^p+\frac1q|B|^q.
\end{equation}
Equivalently, each singular value of $AB^*$ is bounded by the corresponding eigenvalue of the matrix on the right.

Multivariable inequalities of this type are more delicate because the order of the factors matters. \citet{LiuPoonWang2017} obtained a generalized H\"older eigenvalue inequality for products of contractive matrices. More recently, \citet{Lin2024} proved a version of Young's inequality in which the summands on the right may be independently unitarily conjugated. In the same paper, Lin proposed a stronger conjecture for positive semidefinite matrices, with no unitary conjugations on the right. A determinant-level consequence was already known from \citet{Manjegani2007}, but the coordinatewise eigenvalue inequalities were left open.

This paper shows that Lin's conjecture fails for every $m\geq3$. For $m\geq4$, we give a uniform family of $2\times2$ counterexamples consisting of rank-one orthogonal projections. The case $m=3$ is disproved by a $3\times3$ positive semidefinite example for which the second eigenvalue inequality fails. Section~\ref{sec:conjecture} recalls the conjecture and its eigenvalue formulation. Section~\ref{sec:main} gives the counterexamples and their proofs.

\section{The conjecture}\label{sec:conjecture}

For a Hermitian matrix $H$, its eigenvalues are arranged as
\[
 \lambda_1(H)\geq\lambda_2(H)\geq\cdots\geq\lambda_n(H).
\]

The following is Conjecture~5.1 of \citet{Lin2024}.

\begin{conjecture}[Lin]\label{conj:lin}
Let $A_1,\ldots,A_m\in M_n$ be positive semidefinite, and let $p_1,\ldots,p_m>0$ satisfy
\[
 \frac1{p_1}+\cdots+\frac1{p_m}=1.
\]
There exists a unitary matrix $U\in M_n$ such that
\[
 U|A_1\cdots A_m|U^*
 \preceq
 \frac1{p_1}A_1^{p_1}+\cdots+\frac1{p_m}A_m^{p_m}.
\]
Equivalently, for $j=1,\ldots,n$,
\begin{equation}\label{eq:eigen-conjecture}
 \lambda_j(|A_1\cdots A_m|)
 \leq
 \lambda_j\left(
 \frac1{p_1}A_1^{p_1}+\cdots+\frac1{p_m}A_m^{p_m}
 \right).
\end{equation}
\end{conjecture}

When $m=2$, Conjecture~\ref{conj:lin} is precisely Ando's inequality \eqref{eq:ando}, applied to positive semidefinite matrices.

\section{Main results}\label{sec:main}

\begin{lemma}\label{lem:projection}
If $P$ is an orthogonal projection, then $0\preceq P\preceq I$ and $P^t=P$ for every real number $t>0$.
\end{lemma}

\begin{proof}
Since $P=P^*=P^2$, both $P$ and $I-P$ are orthogonal projections. For every vector $x$,
\[
 x^*Px=\|Px\|^2\geq0,
 \qquad
 x^*(I-P)x=\|(I-P)x\|^2\geq0.
\]
Thus $P\succeq0$ and $I-P\succeq0$, which gives $0\preceq P\preceq I$. By the spectral theorem, there is a unitary matrix $V$ such that
\[
 P=V\begin{pmatrix}I_r&0\\0&0\end{pmatrix}V^*.
\]
The spectral calculus gives
\[
 P^t=V\begin{pmatrix}I_r^t&0\\0&0^t\end{pmatrix}V^*=P
\]
for every $t>0$.
\end{proof}

\begin{lemma}\label{lem:cosine}
Let $g(x)=\cos^x\left(\frac{\pi}{2x}\right)$. Then $g$ is strictly increasing for $x>1$.
\end{lemma}

\begin{proof}
For $x>1$,
\[
 \frac{\mathrm{d}}{\mathrm{d}x}\bigl[\log(g(x))\bigr]=\log\left(\cos\left(\frac{\pi}{2x}\right)\right)+\frac{\pi}{2x}\tan\left(\frac{\pi}{2x}\right)=\int_0^{\pi/(2x)}t\sec^2(t)\,\mathrm{d}t>0.
\]
Thus $\log(g(x))$ is strictly increasing, and hence so is $g$.
\end{proof}

\begin{theorem}\label{thm:counterexample}
Conjecture~\ref{conj:lin} is false for every integer $m\geq4$. More precisely, for each such $m$ there are positive semidefinite matrices $A_1,\ldots,A_m\in M_2$ and exponents $p_1,\ldots,p_m>1$ satisfying
\[
 \frac1{p_1}+\cdots+\frac1{p_m}=1
\]
for which
\[
 \lambda_1(|A_1\cdots A_m|)
 >
 \lambda_1\left(
 \frac1{p_1}A_1^{p_1}+\cdots+\frac1{p_m}A_m^{p_m}
 \right).
\]
\end{theorem}

\begin{proof}
Fix $m\geq4$. For $j=1,\ldots,m$, set
\[
 u_j=
 \begin{pmatrix}
 \displaystyle \cos\left(\frac{(j-1)\pi}{2(m-1)}\right)\\[12pt]
 \displaystyle \sin\left(\frac{(j-1)\pi}{2(m-1)}\right)
 \end{pmatrix},
 \qquad
 A_j=u_ju_j^*.
\]
Each $A_j$ is a rank-one orthogonal projection. Choose
\[
 p_1=p_m=\frac52,
 \qquad
 p_j=5(m-2)\quad (2\leq j\leq m-1).
\]
Then $\sum_{j=1}^{m}p_j^{-1}=1$.
By Lemma~\ref{lem:projection}, $A_j^{p_j}=A_j$ for every $j$.

Since $A_j=u_ju_j^*$,
\[
 A_1\cdots A_m
 =\left(\prod_{j=1}^{m-1}u_j^*u_{j+1}\right)u_1u_m^*.
\]
For $1\leq j\leq m-1$, the cosine difference identity gives
{\small
\[
 u_j^*u_{j+1}=\cos\left(\frac{(j-1)\pi}{2(m-1)}\right)\cos\left(\frac{j\pi}{2(m-1)}\right)+\sin\left(\frac{(j-1)\pi}{2(m-1)}\right)\sin\left(\frac{j\pi}{2(m-1)}\right)=\cos\left(\frac{\pi}{2(m-1)}\right).
\]
}
The rank-one matrix $u_1u_m^*$ has largest singular value $1$. It follows that
\[
 \lambda_1(|A_1\cdots A_m|)
 =\cos^{m-1}\left(\frac{\pi}{2(m-1)}\right).
\]
Since $m-1\geq3$, Lemma~\ref{lem:cosine} yields
\[
 \lambda_1(|A_1\cdots A_m|)
 \geq \cos^3\left(\frac{\pi}{6}\right)
 =\frac{3\sqrt3}{8}
 >\frac35.
\]

On the other hand,
\[
 \frac1{p_1}A_1^{p_1}+\cdots+\frac1{p_m}A_m^{p_m}
 =\frac25(A_1+A_m)
 +\frac1{5(m-2)}\sum_{j=2}^{m-1}A_j.
\]
The endpoint vectors are $u_1=(1,0)^T$ and $u_m=(0,1)^T$, so $A_1+A_m=I_2$. Lemma~\ref{lem:projection} also gives $0\preceq A_j\preceq I_2$. Consequently,
\[
 \frac1{p_1}A_1^{p_1}+\cdots+\frac1{p_m}A_m^{p_m}
 \preceq \frac25 I_2
 +\frac{m-2}{5(m-2)}I_2
 =\frac35 I_2.
\]
Thus
\[
 \lambda_1\left(
 \frac1{p_1}A_1^{p_1}+\cdots+\frac1{p_m}A_m^{p_m}
 \right)
 \leq\frac35
 <\lambda_1(|A_1\cdots A_m|),
\]
which contradicts the eigenvalue inequality \eqref{eq:eigen-conjecture}.
\end{proof}

The preceding construction applies when $m\geq4$. The remaining case is settled by the following counterexample.

\begin{theorem}\label{thm:m3}
Conjecture~\ref{conj:lin} is false for $m=3$. More precisely, there are real positive semidefinite matrices $A_1,A_2,A_3\in M_3$ such that, for $p_1=3/2$ and $p_2=p_3=6$,
\[
 \lambda_2(|A_1A_2A_3|)
 >
 \lambda_2\left(\frac23A_1^{3/2}+\frac16A_2^6+\frac16A_3^6\right).
\]
\end{theorem}

\begin{proof}
For a nonzero vector $x\in\mathbb{R}^3$, write
\[
 P_x=\frac{xx^{\mathsf T}}{x^{\mathsf T}x}
\]
for the orthogonal projection onto $\operatorname{span}\{x\}$.
Set
\[
 u=\begin{pmatrix}10\\1\\-2\end{pmatrix},
 \qquad
 v=\begin{pmatrix}0\\2\\1\end{pmatrix},
 \qquad
 w=\begin{pmatrix}10\\1\\0\end{pmatrix}.
\]
Since $u^{\mathsf T}v=0$, the projections $P_u$ and $P_v$ have orthogonal ranges. Define
\[
 A_1=\left(\frac32\right)^{2/3}(25P_u+P_v),
 \qquad
 A_2=6^{1/6}(I_3-P_w),
 \qquad
 A_3=6^{1/6}\operatorname{diag}(100,1,1).
\]
These matrices are real and positive semidefinite.

Because $P_u$ and $P_v$ are orthogonal projections with orthogonal ranges,
\[
 (25P_u+P_v)^{3/2}=125P_u+P_v.
\]
Lemma~\ref{lem:projection} therefore gives
\[
 \frac23A_1^{3/2}+\frac16A_2^6+\frac16A_3^6
 =125P_u+P_v+(I_3-P_w)+\operatorname{diag}(100^6,1,1).
\]
Direct multiplication yields
\[
 A_1A_2A_3=\frac{3}{\sqrt[3]{2}}\,C,
 \qquad
 C=
 {\setlength{\arraycolsep}{10pt}
 \begin{pmatrix}
 0 & 0 & -\dfrac{100}{21}\\[9pt]
 -\dfrac{800}{101} & \dfrac{80}{101} & -\dfrac{8}{105}\\[9pt]
 -\dfrac{400}{101} & \dfrac{40}{101} & \dfrac{121}{105}
 \end{pmatrix}}.
\]
An exact calculation gives
\[
 \det(tI_3-C^{\mathsf T}C)
 =\frac{t\bigl(10605t^2-1094621t+20000000\bigr)}{10605}.
\]
Let
\[
 h(t)=10605t^2-1094621t+20000000.
\]
The two roots of $h$ are the squares of the two nonzero singular values of $C$. At $t=16$,
\[
 h(16)=5200944>0,
 \qquad
 h'(16)=-755261<0.
\]
Since $h$ opens upward, $h'(16)<0$ places $16$ to the left of its vertex; together with $h(16)>0$, this shows that $16$ lies to the left of the smaller root. Hence the second singular value of $C$ is greater than $4$. As $3/\sqrt[3]{2}>2$, we have $\lambda_2(|A_1A_2A_3|)>8$.

The principal submatrix obtained by deleting the first row and column from
$\frac23A_1^{3/2}+\frac16A_2^6+\frac16A_3^6$ is
\[
 B=
 {\setlength{\arraycolsep}{10pt}
 \begin{pmatrix}
 \dfrac{42214}{10605} & -\dfrac{208}{105}\\[9pt]
 -\dfrac{208}{105} & \dfrac{731}{105}
 \end{pmatrix}}.
\]
Moreover,
\[
 8I_2-B=
 {\setlength{\arraycolsep}{10pt}
 \begin{pmatrix}
 \dfrac{42626}{10605} & \dfrac{208}{105}\\[9pt]
 \dfrac{208}{105} & \dfrac{109}{105}
 \end{pmatrix}}.
\]
The upper-left entry of $8I_2-B$ is positive, and
$\det(8I_2-B)=\frac{878}{3535}>0$. Hence $8I_2-B$ is positive definite and $\lambda_1(B)<8$.
Since $B$ is a $2\times2$ principal submatrix of the $3\times3$ Hermitian matrix
$\frac23A_1^{3/2}+\frac16A_2^6+\frac16A_3^6$, the Cauchy interlacing theorem
\citep[Theorem~4.3.17]{HornJohnson2013} gives
\[
 \lambda_2\left(\frac23A_1^{3/2}+\frac16A_2^6+\frac16A_3^6\right)
 \leq\lambda_1(B)<8.
\]
Together with the estimate above, this yields
\[
 \lambda_2(|A_1A_2A_3|)>8>
 \lambda_2\left(\frac23A_1^{3/2}+\frac16A_2^6+\frac16A_3^6\right),
\]
which proves the asserted strict inequality.
\end{proof}

\begin{remark}
Theorems~\ref{thm:counterexample} and~\ref{thm:m3} show that Conjecture~\ref{conj:lin} fails for every $m\geq3$. When $m=2$, it is Ando's inequality~\eqref{eq:ando}.
\end{remark}

\end{document}